\documentclass[12pt]{amsart}
\usepackage[latin1]{inputenc}
\usepackage{latexsym, amsfonts, amsmath, %esint, bbding,
amscd,bezier}
\usepackage[brazil,portuguese,english]{babel}
\usepackage{mathrsfs}    
\usepackage{mathptmx} 
\usepackage{mathtools}  
\usepackage{amsmath,amssymb,amsthm,amsfonts,mathtools}
\usepackage{esint}
\usepackage[usenames,dvipsnames,svgnames,table]{xcolor}
\def\Qed{\ifhmode\unskip\nobreak\fi\quad 
  \ifmmode\square\else$\square$\fi} 
\usepackage[lmargin=2.5cm,tmargin=2cm,bmargin=2cm,rmargin=2.5cm]{geometry}

\usepackage{indentfirst}

\usepackage[bookmarksnumbered,plainpages,backref]{hyperref}

\usepackage{ulem}

\newtheorem{definition}{Definition}[section]
\newtheorem{proposition}{Proposition}[section]
\newtheorem{theorem}{Theorem}[section]

\newtheorem{lemma}{Lemma}[section]

\newtheorem{corollary}{Corollary}[section]

\numberwithin{equation}{section}

\newcounter{remark}[section]
\newenvironment{remark}
{\refstepcounter{remark}\medskip\noindent{\sc Remark\ \thesection.\theremark:}}{\medskip}
\newcounter{alphatheo}
\renewcommand\thealphatheo{\Alph{alphatheo}}
\newenvironment{alphatheo}{\refstepcounter{alphatheo}\medskip\noindent{\bf Theorem \thealphatheo.} \it }{\medskip}

\newcounter{example}

\renewenvironment{proof}{\medskip\noindent{\sc Proof:}}{\medskip}

\newcommand{\R}{\mathbb R}

\newcommand{\N}{\mathbb N}

\newcommand{\C}{\mathbb C}

\def\M{{\mathcal M}}

\def\ccinf{C^\infty_{c}}

\def\<{\langle}
\def\>{\rangle}
\def\eps{\varepsilon}

\def\ds{\displaystyle}

\DeclareRobustCommand{\charchi}{{\mathpalette\irchi\relax}} 
\newcommand{\irchi}[2]{\raisebox{\depth}{$#1\chi$}}

\DeclareMathOperator*{\real}{Re}
\DeclareMathOperator*{\imag}{Im}
\DeclareMathOperator*{\diver}{div}
\DeclareMathOperator*{\esssup}{ess\hspace{2pt}sup}

\begin{document}

\title[Lebesgue solvability of elliptic diff. op. for measures]{A note on Lebesgue solvability of elliptic homogeneous linear equations with measure data}

\author {V. Biliatto}
\address{Departamento de Matem\'atica, Universidade Federal de S\~ao Carlos, S\~ao Carlos, SP,
13565-905, Brazil}
\email{victorbiliatto@estudante.ufscar.br}

%\author {L. Moonens}
%\address{Departamento \!de\! Computa\c{c}\~ao\! e \!Matem\'atica, Universidade de S\~ao Paulo, Ribeir\~ao Preto, SP, 14040-901, Brasil}
%\email{@}

\author {T. Picon}
\address{Departamento de Computa\c{c}\~ao e Matem\'atica, Universidade de S\~ao Paulo, Ribeir\~ao Preto, SP, 14040-901, Brazil}
\email{picon@ffclrp.usp.br}

\thanks{The first was partially supported by Coordena\c{c}\~ao de Aperfei\c{c}oamento de Pessoal de N\'ivel Superior (CAPES - 88882.441243/2019-01) and the second by Conselho Nacional de Desenvolvimento Cient\'ifico e Tecnol\'ogico (CNPq - grant 311430/2018-0) and Funda\c{c}\~ao de Amparo \`a  Pesquisa do Estado de S\~ao Paulo (FAPESP - grant 18/15484-7). }



\begin{abstract}
In this work, we present new results on solvability of the equation $A^{*}(D)f=\mu$ for $f \in L^{p}$ and positive measure data $\mu$ associated to an elliptic homogeneous {linear} differential operator $A(D)$ of order m. Our method is based on $(m,p)-$energy control of $\mu$ giving a natural characterization for solutions when $1\leq p < \infty$. We also obtain sufficient conditions in the limiting case $p=\infty$ using {new $L^{1}$ estimates on measures for elliptic and canceling operators.}  %As application we extend and recover several well know  on measures with peculiar assumption on canceling property where $p=\infty$.      in the class of solutions $F \in L^{p}$, where $1\leq p \leq \infty$.}
%We also present the local solvability ..... Let $A(x,D)$ be an elliptic linear differential operator of order $\nu$ with smooth complex caoefficients in  $\Omega\subset\erre^N$ from a complex vector space $E$ to a complex vector space $F$.In this paper we show that if 
\end{abstract}

\maketitle

%\tableofcontents

%%%%%%%%%%%%%%%1111111111111111
\section{Introduction} 

N. Phuc and M. Torres in \cite{PT} characterized the existence of solutions in Lebesgue spaces for the divergence equation
\begin{equation}\label{div}
\diver f =\nu,
\end{equation}
where $\nu \in \M_{+}(\R^N)$, the set of scalar positive Borel measures on $\R^N$, and $f \in L^{p}(\R^N,\R^N)$. The method is based on controlling the $(1,p)-$ energy of $\nu$ defined by $\|I_{1}\nu\|_{L^{p}}$, where $I_{1}$ is the Riesz potential operator. In fact, $\|I_{1}\nu\|_{L^{p}}$ {finite} is a necessary condition for solvability in $L^{p}$, since from   \eqref{div} we have 
\begin{equation}\label{i1}
I_{1}\nu=c_N\sum_{j=1}^{N}R_{j}f_{j}
\end{equation}
and the control {in norm} follows as a direct consequence of the continuity of Riesz transform operators $R_{j}$ in $L^{p}(\R^{N})$ for $1<p<\infty$. The following result was {proved} in \cite[Theorems 3.1 and 3.2]{PT}:\\

\noindent \textbf{Theorem.}
If $f \in  L^{p}(\R^N,\R^N)$ satisfies \eqref{div} for some $\nu \in \M_{+}(\R^N)$, then
\begin{enumerate}
\item[(i)]$\nu=0$, assuming $1 \leq p \leq N/(N-1)$;
\item[(ii)]$\nu$ has finite $(1,p)$-energy, assuming  $N/(N-1)<p<\infty$. Conversely, if 
$\nu \in \M_{+}(\R^N)$ has finite $(1,p)-$energy, then there is a vector field $f \in  L^{p}(\R^N,\R^N)$ satisfying \eqref{div}.
\end{enumerate}
The previous result does not cover the case $p=\infty$, %\sout{since the identity \eqref{i1} does not give any additional information} 
{since the proof breaks down once the Riesz transform is not bounded in $L^{\infty}(\R^{N})$}. However from Gauss-Green theorem, if $f \in L^{\infty}(\R^{N}, \R^{N})$ is a solution of \eqref{div} then for any ball $B(x,r)$ there exists $C=C(N)>0$ such that
\begin{align}\nonumber
\nu(B(x,r)) &= \int_{\partial B(x,r)} f \cdot n \; d\mathcal{H}^{N-1}
%&\leq \int_{\partial B(x,r)} \|f\|_{L^\infty} \, d\mathcal{H}^{N-1} 
\leq C \|f\|_{L^\infty}r^{N-1}.
\end{align}
It is easy to check that $\|I_{1}\nu\|_{L^{\infty}}<\infty$ is stronger than {previous one}. A non trivial argument (see \cite{PT}) is sufficient to show that control
\begin{align}\label{stronger}
\nu(B(x,r))  
\leq C r^{N-1},
\end{align}
where the constant is independent of $x \in \R^{N}$ and $r>0$ implies that 
\begin{equation}\label{gaussgreen}
\left| \int_{\R^{n}}u(x)d\nu \right| \leq C \|\nabla u \|_{L^{1}}, \quad \forall \, u \, \in C_{c}^{\infty}(\R^{n})
\end{equation}
and from standard duality argument a solution for \eqref{div} {in $f \in L^{\infty}(\R^{N}, \R^{N})$} is obtained. % \textcolor{red}{Here $W^{1,1}(\R^n)$ denotes the closure in $L^{1}$ norm of $\varphi \in C_{c}^{\infty}(\R^{n})$ with $\|\nabla \varphi\|_{L^{1}}$.} % functions  $L^{1}(\R^{n})$ whose derivatives to order one belongs }
{Measures satisfying  the Morrey control for $1 \leq \lambda <\infty$ given by 
$$\|\mu\|_{\lambda}:=\sup_{B}\frac{ |\mu| (B(x,r))}{r^{\lambda}}<\infty, $$
where the supremum is taken for all open balls $B=B(x,r)$ with $x \in \R^{N}$ and $r>0$, and $|\mu|$ is the total variation on $\mu$ are referred as $\lambda-\textit{Ahlfors regular}$.}

%The previous inequality is the full version of the boundedness in $L^{1}$ norm to Hardy operator $Tg(t)=\frac{1}{t}\int_{0}^{t}g(s)ds$  

Let $A(D)$ be a homogeneous linear differential operator on $\R^{N}$, $N\ge2$, with constant coefficients  of order $m$ from a {finite dimensional} {complex} vector space E to {a finite dimensional} {complex} vector space F given by
$$A(D)=\sum_{|\alpha|=m} a_{\alpha}\partial^{\alpha}: C_{c}^{\infty}(\R^{N},E) \rightarrow  C_{c}^{\infty}(\R^{N},F) , \quad a_{\alpha} \in \mathcal{L}(E,F).$$

Inspired by the previous {theorem}, in this paper we carry further the study of Lebesgue solvability for the equation 
\begin{equation}\label{main1}
A^{\ast}(D)f=\mu,
\end{equation}
where $A^{*}(D)$ is the (formal) adjoint operator associated to the homogeneous {linear} differential operator $A(D)$. 
 %\textcolor{red}{The \eqref{localcondition} is a Wollf type condition ...}
Naturally, the concept of energy of the measure $\mu$ associated to \eqref{main1} can be extended in accordance to the order of $A(D)$ {called} $(m,p)$-energy of $\mu$ defined by the functional $\|I_m \mu\|_{L^p}$ (see the Definition \ref{def2} {for complete details}).

Our first result {concerns the Lebesgue solvability for the equation \eqref{main1} when $1 \leq p <\infty$.}

%an extension of previous theorem  \cite[Theorems 3.1 and 3.2]{PT} for higher order homogeneous linear differential operators:} %\sout{ announced at \cite{PT}:}

%Our first result is a slight improvement of the theorems from \cite{PT} stated above:

\begin{alphatheo}\label{theoA}
Let $A(D)$ be a homogeneous linear differential operator of order $1 \leq m<N$ on $\R^{N}$, $N\ge2$, from $E$ to $F$ and  $\mu \in \M_{+}(\R^N,E^{*})$. 
\begin{itemize}
\item[(i)] If  $1 \leq p \leq N/(N-m)$ and $f \in L^p(\R^N,F^{*})$ is a solution for \eqref{main1} then $\mu \equiv 0$.
\item[(ii)] If $N/(N-m) < p < \infty$ and  $f \in L^p(\R^N,F^{*})$ is a solution for \eqref{main1} then $\mu$ has finite $(m,p)-$energy. Conversely, if $|\mu|$ has finite $(m,p)-$energy and $A(D)$ is elliptic, then there exists a function $f \in L^p(\R^N,F^{*})$ solving \eqref{main1}.
\end{itemize}
\end{alphatheo}

We recall that ellipticity means the symbol $A(\xi): E \rightarrow F$ given by
\begin{equation}
A(\xi):=\sum_{|\alpha|=m} a_{\alpha}\xi^{\alpha}
\end{equation}
is injective for $\xi \in \R^{N}\backslash  \left\{ 0 \right\}$.  In particular, the {Theorem A recovers} {\cite[Theorems 3.1 and 3.2]{PT}} taking $A(D)=\nabla$, where $E=\R$ and $F=\R^N$,  % : C^{\infty}(\R^N,\R) \rightarrow C^{\infty}(\R^N,\R^{N})$ 
which is elliptic and $A^{*}(D)=\diver$.%, the divergence operator.

Our second and main result deals with the case $p=\infty$.

\begin{alphatheo}\label{theoB}
Let $A(D)$ be a homogeneous linear differential operator of order $1\leq m<N$ on $\R^{N}$ from $E$ to $F$ and  $\mu \in \M_{+}(\R^N,E^{*})$. % If $F \in L^p(\R^N,E)$ is a solution for $A^{\ast}(D)F = \mu$, then $\mu$ has finite $(m,p)$-energy. Conversely, 
If $A(D)$ is elliptic and cancelling, and $\mu$ satisfies %both
\begin{equation}\label{stronger m}
%\textcolor{red}{|\mu|(B_r)} \leq C_1 r^{N-m},
{\|\mu\|_{0,N-m}:=\sup_{r>0} \frac{|\mu|(B_r)}{ r^{N-m}}}<\infty,
\end{equation}
and the potential control %following truncated Wolff's potential 
%\begin{enumerate}
%\item[(i)] there exists a constant $C>0$ such that $\mu(B(x,r)) \leq C r^{N-m}$ for every $x \in \R^N$ and $r>0$;
%\item[(ii)] $\ds\int_0^{|y|/2} \dfrac{\mu(B(y,r))}{r^{N-m+1}} \, dr < \infty$,
%\end{enumerate}
\begin{equation}\label{localcondition}
\ds\int_0^{{|y|/2}} \dfrac{{|\mu|(B(y,r))}}{r^{N-m+1}} \, dr \lesssim 1, \quad  \text{ uniformly on }y ,
\end{equation}
then there exists $f \in L^\infty(\R^N,F^{*})$ solving \eqref{main1}. 
\end{alphatheo}

 {We point out that the assumption \eqref{stronger m} is weaker in comparison to $\|\mu\|_{N-m}<\infty$, since it is only necessary to take the supremum over balls centered at the origin. The condition \eqref{localcondition} can be understood as an uniform control of the truncated Wolff's potential associated to positive Borel measures on $\R^{N}$ originally defined 
 $$W^{t}_{\alpha,p}\nu(x)=\int_{0}^{t}\left[ \frac{\nu(B(x,r))}{r^{N-\alpha p}} \right]^{\frac{1}{p-1}}\frac{dr}{r}$$
 for $1<p<\infty$ and $\alpha>0$ (see \cite{HW} for original introduction of Wolff's potential and \cite{AdHe, PV} for applications).}
 
 The canceling property means 
 \begin{equation}\label{canceling}
 \displaystyle{\bigcap_{\xi \in \R^{N}\backslash 0}\,A(\xi)[E]=\left\{ 0 \right\}}.
 \end{equation}
{The theory of canceling operators is due to J. Van Schaftingen (see  \cite{VS}), motivated by studies of some $L^{1}$  \textit{a priori}  estimates for vector fields with divergence free and chain of complexes.}% In particular he characterized  of the classical Sobolev-Gagliardo-Nirenberg inequality
%$$ \| D^{m-1} \; u \|_{L^{N/(N-1)}} \leq C \| A(D) u \|_{L^1}, \quad \quad \forall \, u \in C_{c}^{\infty}(\R^{N};E)$$ 
%for elliptic operators. In particular, this inequality unifies several well known \textit{a priori} estimates for vector fields.}

The main ingredient in the proof  Theorem \ref{theoB} is to investigate sufficient conditions on $\mu$ in order to obtain
\begin{equation}\label{ineqmedida}	
\left|\int_{\R^N} u(x) \, d\mu(x) \right| \lesssim 
		 \|A(D)u\|_{L^{1}}, \qquad \forall \, u \in C_{c}^{\infty}(\R^{N},E).
\end{equation}
Inequalities of this type were studied by P. de N\'apoli and T. Picon in \cite{DNP} in the setting of vector fields associated to cocanceling (see the Definition \eqref{de1.1}) operators where $d\mu=|x|^{-\beta}dx$ i.e. the {(scalar) positive} measure is given by {special weighted power for some $\beta>0$}. {More recently, J. Van Schaftingen,  F. Gmeineder and B. Rai\c{t}\u{a} (see \cite[Theorem 1.1]{GRS}) characterized a similar inequality involving positive Borel scalar measures, precisely: if $q=\frac{N-s}{N-1}$ and $0 \leq s <1$ then the estimate }
\begin{equation}\label{raita}
\left(\int_{\R^N} \left|D^{m-1}u(x)\right|^{q} \, d\nu(x) \right)^{1/q}  \lesssim \|\nu \|_{{q(N-1)}}^{1/q} 
		 \|A(D)u\|_{L^{1}}, \quad 
\end{equation}
for all $u \in C_{c}^{\infty}(\R^{N},E)$ {and all $q(N-1)-$Ahlfors regular measure $\nu$,}
%\begin{equation}
%\|\mu \|_{L^{1,\lambda}}=\sup_{B(x,R)} \frac{|\mu|(B(x,R))}{R^{\lambda}}
%\end{equation} 
holds if and only if $A(D)$ is elliptic and canceling. Besides the authors claim that it seems to be no simple a generalization for $s=1$ i.e $q=1$, in particular the inequality holds for the total derivative operator $A(D)=D^{m}$ that is elliptic and canceling {(see Remark \ref{remarkdm})}. %In the Subsection \ref{trace}, we present an extension for \eqref{raita} when $s=1$ for elliptic and canceling  operators assuming an additional potential condition on $\mu$. %       The authors mentioned also discussed parcial results for the case $s=1$. % in the setting of multiplicative inequalities (see).  
{We also point out that in a different fashion from the previous result we obtain sufficient conditions on $\mu$ to fulfil \eqref{ineqmedida}  that come naturally from $(m,p)-$energy control.}

{The paper is organized as follows. In Section \ref{sec2} we briefly study properties of measures with finite $(m,p)$-energy. The proof of Theorem \ref{theoA} is presented in Section \ref{sec3}. The Section \ref{sec4} is devoted to the proof of Theorem \ref{theoB}, where a Fundamental Lemma \ref{main}, with own interest, is presented. Finally in Section \ref{sec5} we present some general comments, in particular we discuss an extension for inequality \eqref{raita} when $s=1$ for elliptic and canceling operators and 
 a reciprocal to Theorem \ref{theoB} for first order operators.}\\

\noindent \textbf{Notation:} throughout this work, the symbol $f \lesssim g$ means that there exists a constant $C>0$, neither depending on $f$ nor $g$, such that $f \leq C \, g$. %By a dyadic cube we mean cubes on $\R^n$, open on the right whose vertices are adjacent points of the lattice $(2^{-k}\Z)^n$ for some $k\in \Z$.
Given a set $A\subset \R^N$ we denote by $|A|$ its Lebesgue measure. We write $B=B(x,R)$ for the open ball with center $x$ and radius $R>0$. By $B_{R}$ we mean the ball centered at the origin with radius $R$. We fix $\fint_{Q}f(x)dx := \frac{1}{|Q|}\int_{Q}f(x)dx$ and  denote $\mathcal{M}_{+}(\Omega, \mathbb{C})$ the set of complex-valued positive Borel measures on $\Omega \subseteq \R^N$ given by $\mu=\mu^{\real} + i\,\mu^{\imag}$, where $ \mu^{\real}, \mu^{\imag} \in \mathcal{M}_{+}(\Omega):= \mathcal{M}_{+}(\Omega, \R)$.

%%%%%%%%%%%%%%%%%%

\section{Measures with finite energy}\label{sec2}

For any $0 < m < N$ and $f \in S(\R^N)$, consider the fractional integrals also called Riesz potential operators defined by
\[I_m f(x) = \dfrac{1}{\gamma(m)} \int_{\R^N} \dfrac{f(y)}{|x-y|^{N-m}} dy,\]
with $\gamma(m) := \pi^{N/2} 2^m \Gamma(m/2)/\Gamma\left({(N-m)}/{2}\right)$. 

{Let $\Omega \subseteq \R^N$ be an open set and $X$ be a complex vector space with {$\dim_{\C} X=d < \infty$}. We denote by $\M_{+}(\Omega,X)$ the set of all $X$-valued complex vector space measures on $\Omega$, $\mu = (\mu_{1},\dots, \mu_{d})$ where        $\mu_{\ell}=\mu_{\ell}^{\real} + i\,\mu_{\ell}^{\imag} \in \mathcal{M}_{+}(\Omega,\C)$ for $\ell=1,\dots,d$. }
{If
 $\eta \in \mathcal{M}_{+}(\Omega,\C)$ {then} we define
\[I_m \eta(x) = \dfrac{1}{\gamma(m)} \int_{\Omega} \dfrac{1}{|x-y|^{N-m}} d\eta(y)\]
and $I_{m}\mu:=(I_{m}\mu_{1},\dots,I_{m}\mu_{d})$ for $\mu \in \M_{+}(\Omega,X)$.
Clearly $|I_{m}\mu(x)|\leq I_{m}|\mu|(x)$, where %$|\eta|$ is the total variation of $\eta$ and 
$|\mu|:=\sum_{j=1}^{d}|\mu_{\ell}|$ {and $|\mu_{\ell}|$ is the total variation of $\mu_{\ell}$ for each $\ell \in \left\{ 1, \dots, d \right\}$.} % we have $|I_{m}\mu(x)|\leq I_{m}|\mu|(x)$.}

%Let $X$ be a finite dimensional complex vector space and $\textcolor{red}{dim_{\C}\,X=d}$. For $\mu \in \M_{+}(\Omega,X)$ we denote each component of $\mu$ by $\mu_{\ell}$  

\begin{definition}\label{def2}
Let $1 \leq p < \infty$ and $0<m<N$. % $\Omega \subseteq \R^N$ an open set and $X$ be a finite dimensional complex vector space. 
We say that $\mu \in \M_{+}(\Omega,X)$ has finite $(m,p)-$energy if
\[ \|I_m \mu\|_{L^{p}}:= \left( \int_{\R^N} |I_m \mu(x)|^p \, dx \right)^{1/p} < \infty,\]
%where \textcolor{red}{for each $\eta \in \mathcal{M}_{+}(\Omega,\C)$} we define
%\[I_m \eta(x) = \dfrac{1}{\gamma(m)} \int_{\Omega} \dfrac{1}{|x-y|^{N-m}} d\eta(y).\]
{and $\mu$ has finite $(m,1)-$weak energy if
\[\|I_m \mu\|_{L^{1,\infty}} \doteq \sup_{\lambda > 0} \lambda \, \left| \{x: |I_m \mu(x)|>\lambda\} \right| <\infty.\]}
\end{definition}

{From previous definition follows $\|I_{m}\mu_{\ell}^{\real}\|_{L^{p}}+\|I_{m}\mu_{\ell}^{\imag}\|_{L^{p}} \lesssim \|I_{m}\mu\|_{L^{p}}$ for $\ell=1,\dots,d$. The same estimate holds replacing $L^{p}$ by $L^{1,\infty}$.}

\begin{proposition}\label{finite Im energy}
If $\mu \in \M_{+}(\Omega,X)$ has finite $(m,p)-$energy for some $1 < p \leq N/(N-m)$ or $(m,1)-$weak energy then $\mu \equiv 0$ on $\Omega$. % For $p=1$, if
	%then, $\mu \equiv 0$ on $\Omega$.
\end{proposition}

\begin{proof} {Let $R>0$ and by simplicity we assume $\mu_{\ell} \in \mathcal{M}_{+}(\Omega)$ for each $\ell \in \left\{ 1, \dots, d \right\}$. We have} % positive for any $\ell \in \{1,\dots,d\}$}. We have
	\begin{align*}
		I_m \mu_{\ell}(x) \gtrsim \int_{B_R \cap \Omega} \dfrac{1}{|x-y|^{N-m}} \, d\mu_{\ell}(y)
		%& \gtrsim \int_{B_R \cap \Omega} \dfrac{1}{(|x|+R)^{N-m}} \, d\mu_{\ell}(y)\\ 
		& \geq \dfrac{\mu_{\ell}(B_R \cap \Omega)}{(|x|+R)^{N-m}}.
	\end{align*}
	Thus,
	\begin{align*}
		\int_{\R^N} |I_m \mu(x)|^p \, dx &\gtrsim \int_{\R^N} [I_m \mu_{\ell}(x)]^p \, dx \geq \int_{\R^N} \left[\dfrac{\mu_{\ell}(B_R \cap \Omega)}{(|x|+R)^{N-m}}\right]^p \, dx\\
		&= \left[\mu_{\ell}(B_R \cap \Omega)\right]^p \int_{\R^N} \dfrac{1}{(|x|+R)^{(N-m)p}} \, dx.
	\end{align*}
	Observe that for $1 < p \leq N/(N-m)$ {the last integral} 
	%\begin{align*}
	%	\int_{\R^N} \dfrac{1}{(|x|+R)^{(N-m)p}} \, dx &%= c(N) \int_{0}^\infty \dfrac{r^{N-1}}{(r+R)^{(N-m)p}} \, dr
		%=  c(N)\int_{R}^\infty \dfrac{(r-R)^{N-1}}{r^{(N-m)p}} \, dr
	%\end{align*}
blows-up to infinity, %, since $N-1 - (N-m)p \geq -1$. 
thus we must have $\mu_{\ell}(B_R \cap \Omega)=0$, since $ \|I_m \mu \|_{L^p}<\infty$.
To the case $p=1$ we have
 %Similarly, $\mu_{\ell}^{\imag}(B_R \cap \Omega)=0$. 
%If $\|I_m \mu\|_{L^{1,\infty}}<\infty$, then, for every $\ell \in \{1,\dots,d\}$,
\[\sup_{\lambda > 0} \lambda \left|\left\{x \in \R^{N}\,:\, \frac{\mu_{\ell}(B_R \cap \Omega)}{(|x|+R)^{N-m}} > \lambda\right\}\right| \lesssim  
\|I_m \mu\|_{L^{1,\infty}}<\infty.\]
%However,
%\begin{align*}
%	\frac{\mu_{\ell}(B_R \cap \Omega))}{(|x|+R)^{N-m}}>\lambda %&\iff |x| < \left(\frac{\mu_{\ell}(B_R \cap \Omega)}{\lambda}\right)^{\frac{1}{N-m}} - R\\ 
%	&\iff x \in B\left(0,\left( \frac{\mu_{\ell}(B_R \cap \Omega)}{\lambda} \right)^{\frac{1}{N-m}}-R\right).
%\end{align*}
Thus,
\begin{align*}
	\lambda\,\left|\left\{x: \frac{\mu_{\ell}(B_R \cap \Omega)}{(|x|+R)^{N-m}}>\lambda\right\}\right|
	&= \lambda\, \left|B\left(0,\left(\frac{\mu_{\ell}(B_R \cap \Omega)}{\lambda}\right)^{\frac{1}{N-m}}-R\right)\right|\\
	%&= \lambda\, \left|\left(\dfrac{c}{\lambda}\right)^{\frac{1}{N-m}}B\left(0,\mu_{\ell}(B_R \cap \Omega)^{\frac{1}{N-m}}-\left({\lambda}\right)^{\frac{1}{N-m}}R\right)\right|\\
	%&= \lambda^{1-\frac{N}{N-m}}\, \left|B\left(0,\mu_{\ell}(B_R \cap \Omega)^{\frac{1}{N-m}}-\lambda^{\frac{1}{N-m}}R\right)\right|\\
	&= {\lambda}^{-\frac{m}{N-m}}\, \left|B\left(0,\mu_{\ell}(B_R \cap \Omega)^{\frac{1}{N-m}}-\lambda^{\frac{1}{N-m}}R\right)\right|,
\end{align*}
which blows-up to infinity when $\lambda>0$ is small and $\mu_{\ell}(B_R \cap \Omega) \neq 0$. %The conclusion follows as before. 
Given that $R>0$ was arbitrarily chosen, and that $\displaystyle{\Omega = \bigcup_{k \in \N} [B_k \cap \Omega]}$, we conclude that $\mu_\ell \equiv 0$ on $\Omega$ for every $\ell \in \{1,\dots,d\}$. Therefore, $\mu \equiv 0$.
\Qed
\end{proof}

\section{Proof of Theorem \ref{theoA}} \label{sec3}

Throughout this section, $A(D)$ denotes an elliptic homogeneous linear differential operator of order $m$ on {$\R^N$}, $N \geq 2$ and $1 \leq m < N$, with constant coefficients from a finite dimensional complex vector space $E$ to a finite dimensional complex vector space $F$. {Since the vector spaces have finite dimension we will use the identification $X$ instead $X^{*}$, for simplicity.}

%\textcolor{red}{The first part of Theorem A is the statement of next result.}

\begin{proposition}\label{p pequeno coef const}
	Let $1 \leq p \leq N/(N-m)$. If $\mu \in \M_{+}(\R^N,E)$ and $f \in L^p(\R^N,F)$ is a solution for $A^*(D)f = \mu$, then $\mu \equiv 0$.
\end{proposition}
\begin{proof}
	From the identity
	$\displaystyle{(N-m)\int_{|x-y|}^{\infty}\frac{1}{r^{N-m+1}} \, dr = \frac{1}{|x-y|^{N-m}}}$
	and the Fubini's theorem we may write 
\begin{align*}
		I_m\mu(x) %= \int_{\R^N}\frac{1}{|x-y|^{N-m}} \, d\mu(y)
		=c_{N,m}\int_{\R^N}\left(\int_{|x-y|}^{\infty}\frac{1}{r^{N-m+1}} \, dr \right)\, d\mu(y) 
	%\end{align*}
%	If ${A_y = \{r \in \R : r>|x-y|\}}$, then by Fubini's Theorem,
%	\begin{align*}
		%&=\int_{\R^N}\left(\int_{|x-y|}^{\infty}\frac{1}{r^{N-m+1}} \, dr \right)\, d\mu(y) \\%&=\int_{\R^N}\left(\int_{0}^{\infty} \frac{\charchi_{A_y}(r)}{r^{N-m+1}} \, dr \right)\, d\mu(y)\\
		&= c_{N,m} \int_{0}^{\infty}\left(\int_{\R^N} \frac{\charchi_{ \left\{r>|x-y|\right\}}  (r)}{r^{N-m+1}} \, d\mu(y) \right)\, dr \\
		%&= \int_{0}^{\infty}\left(\int_{B(x,r)} \frac{1}{r^{N-m+1}} \, d\mu(y) \right)\, dr\\
		%&= \int_{0}^{\infty} \frac{\mu(B(x,r))}{r^{N-m+1}} \, dr\\
		&=c_{N,m} \lim\limits_{\eps \to 0^+}\int_{\eps}^{\infty} \frac{\mu(B(x,r))}{r^{N-m+1}} \, dr.
		%&= \lim\limits_{\eps \to 0^+}\int_{\eps}^{\infty} \frac{1}{r^{N-m+1}} \left(\int_{B(x,r)} A^*(D)f(y) \,dy \right)\, dr.
	\end{align*}
Now, using the Gauss-Green theorem, we have
\begin{align*}	
	\mu(B(x,r)) &=\int_{B(x,r)} A^*(D)f(y) \,dy = \sum_{|\alpha|=m} a^{*}_\alpha \int_{B(x,r)} \partial^\alpha f(y) \, dy\\
	&=  \sum_{|\alpha|=m} a^{*}_\alpha \int_{\partial B(x,r)} \partial^{\alpha-e_{j_\alpha}} f(y) \, \frac{y_{j_\alpha}-x_{j_\alpha}}{|y-x|} \, d\omega(y),
\end{align*}	
where we choose, for each multi-index $\alpha=(\alpha_1,\dots,\alpha_N)$, a number $j_\alpha \in \{1,\dots,N\}$ such that $\alpha_{j_\alpha} \neq 0$ in a way that $\partial^\alpha f = \partial_{x_{j_\alpha}}(\partial^{\alpha-e_{j_\alpha}} f)$. %We then use Gauss-Green Theorem to get
	%\[\int_{B(x,r)} A^*(D)f(y) \,dy = \sum_{|\alpha|=m} a^{*}_\alpha \int_{\partial B(x,r)} \partial^{\alpha-e_{j_\alpha}} f(y) \, \frac{y_{j_\alpha}-x_{j_\alpha}}{|y-x|} \, dS(y).\]
Summarizing 
	\begin{align*}
		I_m \mu(x) 
		%&= (N-m)\lim\limits_{\eps \to 0^+} \int_{\eps}^{\infty} \frac{1}{r^{N-m+1}} \left(\sum_{|\alpha|=m} a^{*}_\alpha \int_{\partial B(x,r)} \partial^{\alpha-e_{j_\alpha}} F(y) \, \frac{y_{j_\alpha}-x_{j_\alpha}}{|y-x|} \, dS(y)\right)\, dr\\
		&= c_{N,m}\sum_{|\alpha|=m} a^{*}_\alpha \lim\limits_{\eps \to 0^+} \int_{\eps}^{\infty} \left( \int_{|x-y|=r} \partial^{\alpha-e_{j_\alpha}} f(y) \, \frac{y_{j_\alpha}-x_{j_\alpha}}{|y-x|^{N-m+2}} \, d\omega(y)\right)\, dr\\
		&= c_{N,m} \sum_{|\alpha|=m} a^{*}_\alpha \lim\limits_{\eps \to 0^+} \int_{|x-y|>\eps} \partial^{\alpha-e_{j_\alpha}} f(y) \, \frac{x_{j_\alpha}-y_{j_\alpha}}{|x-y|^{N-m+2}} \, dy\\
		&= c_{N,m} \sum_{|\alpha|=m} a^{*}_\alpha \left( K_{j_\alpha} \ast \partial^{\alpha-e_{j_\alpha}} f\right)(x),
\end{align*}
where $K_{j_\alpha}(x) := x_{j_\alpha}/|x|^{N-m+2}$.
Thus from \cite[p.~73]{S1} we have $\widehat{K}_{j_{\alpha}}(\xi) = c_{N,m} \xi_{j_\alpha}/|\xi|^m$ and hence, recalling the constant $c_{N,m}$, we have
	\begin{align*}
		(K_{j_{\alpha}}*\partial^{\alpha-e_{j_\alpha}}&f)\,\widehat{}\,(\xi) %= \widehat{K}(\xi) \widehat{\partial^{\alpha-e_{j_\alpha}}F}(\xi) 
		= c_{N,m} \frac{\xi_{j_\alpha}}{|\xi|^m} \xi^{\alpha-e_{j_\alpha}} \widehat{f}(\xi)
		= c_{N,m}\frac{\xi^{\alpha}}{|\xi|^{m}} \widehat{f}(\xi) 
		\doteq \widehat{(R^\alpha f)}(\xi)
	\end{align*}
where $R^\alpha := R_1^{\alpha_1} \circ R_2^{\alpha_2} \circ \cdots \circ R_N^{\alpha_N}$ is the $\alpha$ order Riesz transform operator. In this way,
\begin{equation}\label{riesz}
I_m \mu = c_{m,N} \sum_{|\alpha|=m} a^{*}_\alpha R^\alpha f.    	
\end{equation}
In particular for $m=1$,
	\begin{align*}
		I_1 \mu(x) &= c_{N}\sum_{j=1}^N a_j^{*} \lim\limits_{\eps \to 0^+} \int_{|x-y|>\eps} f(y) \, \frac{x_{j}-y_{j}}{|x-y|^{N+1}} \, dy
		= c_{N}\sum_{j=1}^N a_j^{*} R_j f(x)
	\end{align*}
	for almost every $x \in \R^N$, where $R_j$ is the $j^{\text{th}}$ Riesz transform operator.% However, for $m \geq 2$, we use the following result found in \cite[p.~73]{S1}:
	%\begin{lemma}\label{transformada}
	%	Let $\gamma \in \C$ such that $0 < \real(\gamma) < N$ and $P(x)$ be a homogeneous harmonic polynomial of degree $k$ in $\R^N$, i.e. $P(x) = \sum_{|\beta|=k} c_\beta x^\beta$ and $\Delta P = 0$. If $K(x) = P(x)/|x|^{N+k-\gamma}$ then $\widehat{K}(\xi) = c_{k,\gamma} P(\xi)/|\xi|^{k+\gamma}$. \textcolor{red}{(organizar a escrita)}
	%\end{lemma}
	%The kernel $K(x) = x_{j_\alpha}/|x|^{N-m+2}$ satisfies the previous lemma with $P(x)=x_{j_\alpha}$, $k=1$ and $\gamma=m-1$. Thus $\widehat{K}(\xi) = c_{m} \xi_{j_\alpha}/|\xi|^m$. Hence,
	%\begin{align*}
	%	(K*\partial^{\alpha-e_{j_\alpha}}&F)\,\widehat{}\,(\xi) = \widehat{K}(\xi) \widehat{\partial^{\alpha-e_{j_\alpha}}F}(\xi) = c_{m} \frac{\xi_{j_\alpha}}{|\xi|^m} \xi^{\alpha-e_{j_\alpha}} \widehat{F}(\xi)\\
	%	&= c_{m}\frac{\xi^{\alpha}}{|\xi|^{m}} \widehat{F}(\xi) \\
	%	&\doteq c_{m} \widehat{(R^\alpha F)}(\xi)
	%\end{align*}
%where $R^\alpha := R_1^{\alpha_1} \circ R_2^{\alpha_2} \circ \cdots \circ R_N^{\alpha_N}$ is the higher order Riesz operator.   

	Since each $R_{j}$ is bounded from $L^p$ to itself for $1<p<\infty$ and of type weak$(1,1)$, we conclude that $\|I_m \mu\|_{L^p} \lesssim \|f\|_{L^p} < \infty$ that is $\mu$ has finite $(m,p)-$energy for $1<p\leq N/(N-m)$ and $\|I_m \mu\|_{L^{1,\infty}} \lesssim \|f\|_{L^1} < \infty$ for $p=1$. It follows from Proposition \ref{finite Im energy} that $\mu \equiv 0$ in $\R^N$.
	\Qed
\end{proof}

Next we prove the second part of the Theorem A.

\begin{proposition}\label{p grande coef const}
	Let $N/(N-m) < p < \infty$ and $\mu \in \M_{+}(\R^N,E)$. If $f \in L^p(\R^N,F)$ is a solution for $A^{\ast}(D)f = \mu$, then $\mu$ has finite $(m,p)$-energy. Conversely, if $|\mu|$ has finite $(m,p)-$energy, then there exists a function $f \in L^p(\R^N,F)$ solving $A^{\ast}(D)f = \mu$.
\end{proposition}

\begin{proof}
{The first part follows from identity \eqref{riesz} and the boundedness of $\alpha$ order Riesz transform operators}. For the converse
consider the function $\xi\mapsto H(\xi)\in \mathcal{L}(F,E)$ defined by
\[H(\xi)=(A^{\ast}\circ A)^{-1}(\xi)A^{\ast}(\xi) \]
that is smooth in $\R^{N}\backslash \left\{0\right\} $ and homogeneous of 
degree $-m$. Here $A^{\ast}(\xi)$ is the symbol of the adjoint operator $A^{\ast}(D)$.
Since we are assuming that $1 \leq m<N$ then  $H$ is a locally integrable  tempered distribution and its inverse Fourier transform $K(x)$
%\[
%K\doteq\F^{-1}H,  \quad \widehat K=H,
%\] 
is a locally integrable tempered distribution homogeneous of degree $-N+m$ (see \cite[p.~71]{Javier}) that satisfies
\begin{equation}\label{decfrac}
u(x)=\int_{\R^{N}}K(x-y)[A(D)u(y)]\,dy, \quad u \in C_{c}^{\infty}(\R^{N},E). 
\end{equation}
and clearly  $|u(x)|\leq  I_m |A(D)u|(x)$.
	Let $w^{m,p'}_A(\R^N,E)$ be the closure of $\ccinf(\R^N,E)$ with respect to the norm $\|u\|_{m,p'} \doteq \|A(D)u\|_{L^{p'}}$. Thus %for each $\ell \in \{1,\dots,d\}$, assuming $\mu_{\ell} \in \mathcal{M}_{+}(\R^{N})$ by simplicity, we have %By \eqref{ineq riesz},
	\begin{align*}
		\left|\int_{\R^N} u(x) \, d\mu(x) \right| \lesssim \int_{\R^N} \left[\int_{\R^N} \dfrac{|A(D)u(y)|}{|x-y|^{N-m}} \, dy\right] d|\mu|(x)
		%&\lesssim \int_{\R^N} |A(D)u(y)| \left[\int_{\R^N} \dfrac{1}{|x-y|^{N-m}} \, d|\mu|(x)\right] dy\\
		&\lesssim  \int_{\R^N} |A(D)u(y)| \, I_m |\mu|(y) \, dy\\
		%&\leq \|A(D)u\|_{L^{p'}} \, \|I_m |\mu|\|_{L^p}\\
		& \leq \|u\|_{m,p'} \, \|I_m |\mu|\|_{L^p} \lesssim \|u\|_{m,p'},
	\end{align*}
	since $|\mu|$ has finite $(m,p)$-energy %Taking all the components  %The previous control follows analogously for $\mu_{\ell}^{Re}$ and $\mu_{\ell}^{Im}$.  As
%\begin{equation}
%\left|\int_{\R^N} u(x) \, d\mu(x) \right| \lesssim  \|u\|_{m,p'}, \quad \quad u \in C_{c}^{\infty}(\R^{N},E) 
%\end{equation}
following that $\mu \in [w^{m,p'}_A(\R^N,E)]^*$. Since $A(D): w^{m,p'}_A(\R^N,E) \to L^{p'}(\R^N,F)$ is a linear isometry, hence its adjoint $A^*(D): L^{p}(\R^N,F) \to [w^{m,p'}_A(\R^n,E)]^*$ is surjective. Therefore, there exists $f \in L^{p}(\R^N,F)$ such that $A^*(D)f = \mu$. \Qed
\end{proof}

\begin{corollary}
Let $N/(N-m) < p < \infty$ and $E,F$ finite dimensional real vector spaces. Then $\mu \in \mathcal{M}_{+}(\R^{N},E)$
has finite $(m,p)-$energy if and only if  there exists a function $f \in L^p(\R^N,F)$ solving $A^{\ast}(D)f = \mu$.
\end{corollary}

\section{Proof of Theorem \ref{theoB}}\label{sec4}

%We now proceed to obtain a $L^\infty$ solution for \eqref{main1}. 

In order to prove Theorem \ref{theoB} it is enough to show  that
\eqref{ineqmedida} holds.
 %, under its hypothesis, we have
%\begin{equation*}
%	\left|\int_{\R^N} u(x) \, d\mu(x) \right| \lesssim \|A(D)u\|_{L^1}, \quad \quad \forall \,\, u \in \ccinf(\R^N,E).
%\end{equation*}
In fact, assuming the validity of this inequality we conclude that $\mu \in [w^{m,1}_A(\R^N,E)]^*$ and \textit{bis in idem} the argument used in the proof of Proposition \ref{p grande coef const} there exists $f \in L^\infty(\R^N,F)$ such that $A^*(D)f = \mu$.
%In order to prove ** 
From the identity \eqref{decfrac}, since $A(D)$ is elliptic, the inequality \eqref{ineqmedida} is equivalent to  
\begin{equation}\label{eqmain}
	\left|\int_{\R^N}  \left[\int_{\R^N}K(x-y)g(y) \, dy\right] \, d\mu(x) \right| \lesssim \|g\|_{L^1}
\end{equation}
 where $g:=A(D)u$ for all  $u \in \ccinf(\R^N,E)$ and moreover
\begin{equation} \label{ka} 
 |K(x-y)| \leq C \; |x - y|^{m-N} , \quad x \neq y 
\end{equation}
and 
\begin{equation} \label{kb} 
 |\partial_y K(x-y)| \leq C \; |x - y|^{m-N-1} , \quad 2|y| \leq |x|.
\end{equation}
 The proof reduces to obtaining inequality \eqref{eqmain} invoking a special class of vector fields in $L^1$ norm associated to an elliptic and canceling operator $A(D)$ {and $\mu$ satisfying \eqref{stronger m} and \eqref{localcondition}. } %and an appropriated decay of $\mu$ on balls depending on the radius and on the order of the operator \textcolor{red}{and a special potential...} 
  
The first step is an extension of Hardy-type inequality \cite[Lemma 2.1]{DNP} on {two} measures, which we present for the sake of completeness.
\begin{lemma}\label{weight1}
Let $1 \leq q <\infty$ and $\nu$ be a $\sigma$-finite real positive measure. Suppose $\tilde{u}$ and $\tilde{v}$ are measurable and nonnegative
%and finite
almost everywhere. Then
\begin{equation}\label{w1}
\left[\int_{\R^N} \left( \int_{B_{|x|/2}} \tilde{g}(y) \, dy \right)^q \tilde{u}(x) \, d\nu(x)\right]^{1/q} \lesssim \int_{\R^N} \tilde{g}(x) \tilde{v}(x) \, dx
\end{equation}
holds for all $\tilde{g} \geq 0$  if and only if 
\begin{equation}\label{w2}
C := \sup_{R>0} \left(\int_{(B_R)^c} \tilde{u}(x) \, d\nu(x)\right)^{1/q} \left(\sup_{x \in B_R} [\tilde{v}(x)]^{-1}\right) < \infty.
\end{equation}
%Analogously
%\begin{equation}\label{w3}
%\left[ \int_{\R^{N}} \left(\int_{(B_{|x|/2})^c} \tilde{g}(y) \, dy \right)^q \tilde{u}(x) \, d\nu(x) \right]^{1/q} \lesssim \int_{\R^N} \tilde{g}(x)\tilde{v}(x)dx
%\end{equation}
%holds for all $\tilde{g} \geq 0$  if and only if 
%\begin{equation}\label{w4}
%\tilde{A}:=\sup_{R>0}\left( \int_{B_R} \tilde{u}(x) \, d\nu(x) \right)^{1/q}\left( \sup_{x \in (B_R)^c} [\tilde{v}(x)]^{-1} \right) < \infty.
%\end{equation}
\end{lemma}

\begin{proof}
%First we prove \eqref{w2} implies \eqref{w1}. 
By Minkowski inequality we have
\begin{align*}
	\left[\int_{\R^N} \left( \int_{B_{|x|/2}} \tilde{g}(y) \, dy \right)^q \tilde{u}(x)  \, d\nu(x)\right]^{1/q}
	%&= \left[\int_{\R^N} \left( \int_{\R^N} \tilde{g}(y) \, \charchi_{\{2|y|<|x|\}}(x,y) \, dy \right)^q \tilde{u}(x) \, d\nu(x)\right]^{1/q}\\
	%&\leq \int_{\R^N} \left( \int_{\R^N} [\tilde{g}(y)]^q \, \charchi_{\{2|y|<|x|\}}(x,y) \, \tilde{u}(x) \, d\nu(x) \right)^{1/q} \, dy\\
	& \leq \int_{\R^N} \tilde{g}(y) \, \left( \int_{(B_{2|y|})^c} \tilde{u}(x) \, d\nu(x) \right)^{1/q} \, dy\\
	& \leq C \int_{\R^N} \tilde{g}(y) \, \tilde{v}(y) \, dy,
\end{align*}
since
\begin{align*}
	\left( \int_{(B_{2|y|})^c} \tilde{u}(x) \, d\nu(x) \right)^{1/q} [\tilde{v}(y)]^{-1} &\leq \left( \int_{(B_{2|y|})^c} \tilde{u}(x) \, d\nu(x) \right)^{1/q} \left(\sup_{x \in B_{2|y|}} [\tilde{v}(x)]^{-1}\right)
	\leq C.
\end{align*}

%For \eqref{w1} implies \eqref{w2}, consider

 Conversely, for $R>0$ consider $S(R) := \ds\esssup_{z \in B_R} [\tilde{v}(z)]^{-1}$. For each $n \in \N$, we define the set 
$\widetilde{M_n} := \left\{z \in B_R : [\tilde{v}(z)]^{-1} > S(R) - \dfrac{1}{n}\right\}$.
From the definition follows $|\widetilde{M_n}|>0$, hence there exist $M_n \subseteq \widetilde{M_n}$ with $0 < |M_n| < \infty$. Choosing  $\widetilde{g}(y) = \charchi_{M_n}(y)$ and using \eqref{w1},we have
\begin{align*}
	\left(\int_{(B_{2R})^c} \tilde{u}(x) \, d\nu(x)\right)^{1/q}  %&= \left[\int_{(B_R)^c} \left(\int_{B_{R}} \charchi_{M_n}(y) \, dy\right)^q \, \tilde{u}(x) \, d\nu(x)\right]^{1/q}\\
	\leq \frac{1}{ |M_n| } \left[\int_{\R^N} \left(\int_{B_{|x|/2}} \charchi_{M_n}(y) \, dy\right)^q \, \tilde{u}(x) \, d\nu(x)\right]^{1/q}
	&\lesssim \fint_{M_n} \tilde{v}(x) \, dx \\
	&\lesssim  \left(S(R)-\dfrac{1}{n}\right)^{-1}.
\end{align*}
Taking $n \to \infty$ we get
%\begin{align*}
$	\displaystyle{\left(\int_{(B_{2R})^c} \tilde{u}(x) \, d\nu(x)\right)^{1/q} S(R) \lesssim 1}$
%\end{align*}
and the result follows since the control is uniform on $R>0$.
%The proof is analogous for \eqref{w3} $\iff$ \eqref{w4}. 
\Qed
\end{proof}

One fundamental property of elliptic and canceling operators $A(D)$ is the existence of a homogeneous linear differential operators $L(D): C^{\infty}(\R^N,F) \rightarrow C^{\infty}(\R^N,V)$ of order $\kappa$ for some finite dimensional complex vector space $V$ such that %$L(D)\circ A(D)=0$ (see ***).  
\begin{equation}\label{cocanceling}
\displaystyle{\bigcap_{\xi \in \R^{N}\backslash 0}\,ker \,L(\xi)\,=\,\bigcap_{\xi \in \R^{N}\backslash 0}\,A(\xi)[E]=\left\{ 0 \right\}}.
\end{equation}
The next definition was also introduced by Van Schaftingen in \cite{VS}:
\begin{definition}\label{de1.1} %\cite{VS}
 Let $L(D)$ be a homogeneous linear differential operators of order $\kappa$ on $\R^{N}$ from $F$ to $V$.
The operator $L(D)$ is cocanceling if 
$$\displaystyle{\bigcap_{\xi \in \R^{N}\backslash \left\{ 0 \right\}}ker\,L(\xi)=\left\{ 0 \right\}}.$$ 
\end{definition}
An example of cocanceling operator on $\R^N$ from $F=\R^N$ to $V=\R$ is the divergence operator $L(D)=\diver$. Indeed, for every $e \in \R^N$ we have $L(\xi)[e]=\xi \cdot e$ and then clearly 
$$\bigcap_{\xi \in \R^{N}\backslash \left\{ 0 \right\}} ker\,L(\xi)=\bigcap_{\xi \in \R^{N}\backslash \left\{ 0 \right\}} \xi^{\perp}=\left\{ 0 \right\}.$$

%\medskip

%Our main result in this paper is a characterization of the Stein-Weiss inequality when $p=1$ for vector fields associated to the kernel of some cocanceling operators.% with special range .

 The following peculiar estimate for vector fields belonging to the kernel of some cocanceling operator was {presented} at \cite[Lemma~3.1]{DNP}.
     \begin{lemma}\label{mainlemma2}
 Let $L(D)$ be a homogeneous linear differential operators of order m on $\R^{N}$ from $F$ to $V$. Then 
%be a cocanceling operator as Definition \ref{de1.1}.   
%be an homogeneous linear differential operator of order $m$ on $\R^N$ from $F$ to $G$. 
there exists $C>0$ %and $m \in \Z^{*}_{+}$ 
such that for every $\varphi \in C_{c}^{m}(\R^{N};F)$, %that satisfies $|x|^{j}|D^{j}\varphi(x)| \in L^{1}_{loc}(\R^{N})$ for $j=0,....,m$ and for all $f \in C_{c}^{\infty}(\R^{N};F)$ such that $L(D)f=0$, 
we have 
\begin{equation}\label{main23}
\left| \int_{\R^N} \varphi(y)\cdot f(y)\,dy\right|\le 
C\sum_{j=1}^m \int_{\R^N}|f(y)|\,|y|^j\, |D^j\varphi (y)|\,dy 
\end{equation}
%\textcolor{red}{for all f %$f \in C^{\infty}(\R^{N};F)$ 
%such that  $L(D)f=0$ in the sense of distribution}.
for all functions $f \in L^{1}(\R^N;F)$ satisfying $L(D)f=0$  in the sense of distributions.
\end{lemma}

The second step to obtain \eqref{eqmain} is an improvement of \cite[Lemma 3.2]{DNP} and \cite[Lemma 2.1]{HP} in the setting of positive {Borel} measures. 

\begin{lemma}\label{main}
Assume $N \geq 2$, $0 < \ell < N$ and $K(x,y) \in L^{1}_{loc}(\R^N \times \R^N, \mathcal{L}(F;V))$ 
%a locally integrable function in $\R^N \times \R^N$ 
satisfying
\begin{equation} \label{kc} 
	|K(x, y)| \leq C \; |x - y|^{\ell-N} , \quad x \neq y 
\end{equation}
and 
\begin{equation} \label{kd}
	|K(x,y)-K(x,0)| \leq C \; \frac{|y|}{|x|^{N-\ell+1}} , \quad 2|y|\leq |x|.
\end{equation}
Suppose {$1 \leq q < \infty$} and let $\nu \in \mathcal{M}_{+}(\R^{N})$ satisfying %$0 \leq \alpha<1$, $\beta<N/q$, $\alpha+\beta>0$ and $\displaystyle{{1}/{q}=1+ ({\alpha+\beta-\ell})/{N}}$
\begin{eqnarray}\label{mu regularity}
	%\nu(B_R) \leq C_1 R^{(N-\ell)q}, \quad R>0,\\
{\|\nu\|_{0,(N-\ell)q} < \infty,}
\end{eqnarray}  
and the following uniform {potential} condition
\begin{equation} \label{mu integral}
 [[\nu]]_{(N-\ell)q}:=\sup_{y \in \R^{N}} \int_0^{{|y|/2}} \dfrac{\nu(B(y,r))}{r^{(N-\ell)q+1}}dr  < \infty.  %\, dr \leq C_2, \quad {\text{ uniformly on } y.}
\end{equation}
%\begin{equation} \label{mu x regularity}
%	\textcolor{red}{\nu(B(x,R)) \leq C_2 \, |x|^{(N-\ell+\alpha)q-N} R^N \quad \text{ when } R < \eps|x|}
%\end{equation}
%for some $0<\eps<1$, where $C_1$ and $C_2$ are positive constants.
If $L(D)$ is cocanceling then there exists $\widetilde{C}>0$ such that 
\begin{equation}\label{mainineq}
	\left( \int_{\R^N} \left| \int_{\R^N} K(x,y)g(y) \; dy \right|^q d\nu(x)  \right)^{1/q} 
\leq \widetilde{C} \int_{\R^{N}} |g(x)| \, dx, 
\end{equation}
for all $g \in L^{1}(\R^{N};F)$ satisfying $L(D)g=0$  %such that  %$f \in L^{1}(\R^N;F)$ such that 
%$L(D)f=0$ 
in the sense of distributions. %Conversely, if for all $f \in C^\infty_c(\R^N;F)$ satisfying $L(D)f=0$ the inequality \eqref{main5} holds then $L(D)$ has to be cocanceling. 
\end{lemma}

\begin{remark}\label{remarkhp}
A stronger condition satisfying  \eqref{mu integral} is given by
\begin{equation} \label{mu x regularity}
	\nu(B(y,R)) \leq C_2 \, |y|^{(N-\ell)q-N} R^N 
\end{equation}
 when {$R < |y|/2$}. The integration boundary $|y|/2$ in \eqref{mu integral} can be swapped to $a|y|$, where $a$ is a fixed constant $0<a<1$. In this case, \eqref{mu x regularity} must hold for $R < a|x|$ to imply \eqref{mu integral}.
\end{remark}}

Let us present an example of positive measures satisfying \eqref{mu regularity} and \eqref{mu integral}.
Suppose $N\geq 2$, $0<\ell<N$, $1\leq q \leq N/(N-\ell) $ and define % $0 \leq  \beta<N/q$ and $\displaystyle{{1}/{q}=1+ {(\beta-\ell)/}{N}}$. 
$d\nu = |x|^{q(N-\ell)-N} dx$. %, where
%$\beta < N/q$,
%$0<\alpha+\beta \textcolor{red}{ \leq \ell}$ \textcolor{red}{(necess\'ario se queremos garantir que $q \geq 1$)} and $1/q = 1+(\alpha+\beta-\ell)/N$, 
The control \eqref{mu regularity} is obvious for the case when $q=N/(N-\ell)$, since $\nu$ is simply the Lebesgue measure and $(N-\ell)q = N$. Otherwise,
\begin{align*}
	\nu(B_R) &= \int_{B_R} |x|^{q(N-\ell)-N} \, dx \lesssim \int_0^R r^{q(N-\ell)-1} \, dr
	\lesssim R^{(N-\ell)q}.
\end{align*}
For \eqref{mu integral} we note that if  $y \in B_R$ and {$R < |x|/2$ then $|x|/2 <
%|x|-R < |x|-|y| \leq
|x+y|
%\leq |x|+|y| < |x|+R
< 3|x|/2$}
thus
\begin{align*}
	\nu(B(x,R)) = \int_{B_R} |x+y|^{q(N-\ell)-N} \, dy \lesssim |x|^{(N-\ell)q-N}R^{N}.
\end{align*}
%\begin{align*}
%	\nu(B(x,R)) \leq C_{\eps,\beta,q} \, |x|^{-\beta q} \left| B_R \right| = C_{\eps,\beta,q,N} \, |x|^{(N-\ell+\alpha)q-N} \, R^N.
%\end{align*}

%\textcolor{red}{parei aqui Tiago} \\

In order to prove the inequality \eqref{eqmain}, and consequently the Theorem B, we estimate 
\begin{equation}\nonumber
	\left|\int_{\R^N}  \left[\int_{\R^N}K(x-y)g(y) \, dy\right] \, d\mu(x) \right| \leq
	\int_{\R^N}  \left|\int_{\R^N}K(x-y)g(y) \, dy \right| \, d|\mu|(x) 
\end{equation}
and we apply the Lemma \ref{main} for $q=1$ and $\nu=|\mu|$, taking $K(x,y)=K(x-y)$ given by identity \eqref{decfrac} that satisfies  \eqref{ka} and \eqref{kb}. %as each component $\mu_{j}^{\real}$ and $\mu_{j}^{\imag}$ for $j=1,\dots,d$. 
Note that  \eqref{mu regularity} and \eqref{mu integral} come naturally from \eqref{stronger m} and \eqref{localcondition}. %, and moreover the kernel $K$ satisfies \eqref{ka} and \eqref{kb}. 
The conclusion follows taking $g:=A(D)u$ that belongs to the kernel of some cocanceling operator $L(D)$ from \eqref{cocanceling}, since $A(D)$ is canceling. 

Now we present the proof of Lemma \ref{main}.

\begin{proof}
Let $\psi \in \ccinf(B_{1/2},\R)$ be a cut-off function such that $0 \leq \psi \leq 1$, $\psi \equiv 1$ on $B_{1/4}$, and write $K(x,y) = K_1(x,y)+K_2(x,y)$ with $K_1(x,y) = \psi(y/|x|)K(x,0)$. We claim that
\begin{equation}\label{Ji}
J_j \doteq \left(\int_{\R^N} \left| \int_{\R^N} K_j(x,y) g(y) \, dy \right|^q d\nu(x)\right)^{1/q} \lesssim \int_{\R^N} |g(x)| \, dx
\end{equation}
for $j=1,2$ and $g \in L^1(\R^N;F)$ satisfying $L(D)g=0$  
in the sense of distributions.  

Using the control \eqref{kc} we may estimate
%To prove our claim we first focus on $J_1$:
\begin{align*}
J_1 &= \left(\int_{\R^N} \left| \int_{\R^N} \psi\left(\dfrac{y}{|x|}\right) g(y) \, dy \right|^q |K(x,0)|^q \, d\nu(x)\right)^{1/q}\\
&\lesssim \left(\int_{\R^N} \left| \int_{\R^N} \psi\left(\dfrac{y}{|x|}\right) g(y) \, dy \right|^q |x|^{(\ell-N)q} \, d\nu(x)\right)^{1/q}\\
&\lesssim \left(\int_{\R^N} \left[ \int_{B_{|x|/2}} \dfrac{|y|}{|x|} |g(y)| \, dy \right]^q |x|^{(\ell-N)q} \, d\nu(x)\right)^{1/q}\\
&= \left(\int_{\R^N} \left[ \int_{B_{|x|/2}} |y| |g(y)| \, dy \right]^q |x|^{(\ell-N-1)q} \, d\nu(x)\right)^{1/q},
\end{align*}
where the second inequality follows from \eqref{main23}. In order to control the previous term we use the first part of Lemma \ref{weight1}, taking $\tilde{u}(x)=|x|^{(\ell-N-1)q}$, $\tilde{g}(x)=|x||g(x)|$ and $\tilde{v}(x)=|x|^{-1}$. So checking \eqref{w2} we have
%\begin{align*}
%	\sup_{x \in B_R} [\tilde{v}(x)]^{-1} = \sup_{x \in B_R} |x|^{1-\alpha} = R^{1-\alpha}
%\end{align*}
%and
\begin{align*}
	\left(\int_{(B_R)^c} \tilde{u}(x) \, d\nu(x)\right)^{1/q} %&:= \left(\int_{(B_R)^c} |x|^{(\ell-N-1)q} \, d\nu(x)\right)^{1/q}\\
	&{=} \left(\sum_{k=1}^\infty \int_{2^{k-1}R \leq |x| < 2^{k}R} |x|^{(\ell-N-1)q} \, d\nu(x)\right)^{1/q}\\
	&\leq \left(\sum_{k=1}^\infty (2^{k-1}R)^{(\ell-N-1)q} \nu(B_{2^{k}R})\right)^{1/q}\\
	&\leq {\|\nu\|^{1/q}_{0,(N-\ell)q}} \left(\sum_{k=1}^\infty (2^{k-1}R)^{(\ell-N-1)q} (2^{k}R)^{(N-\ell)q}\right)^{1/q}\\
	%&=C_1^{1/q} \, R^{-1} \, 2^{N-\ell+1} \left(\sum_{k=1}^\infty 2^{-kq}\right)^{1/q}\\
	%&=C_1^{1/q} \, R^{-1} \, 2^{N-\ell+1} \left(2^q-1\right)^{-1/q}\\
	& {\lesssim \|\nu\|^{1/q}_{0,(N-\ell)q}} \left\{\sup_{x \in B_R} [\tilde{v}(x)]^{-1}\right\}^{-1},
\end{align*}
where the last step follows from $\displaystyle{\sup_{x \in B_R} [\tilde{v}(x)]^{-1} = R}$. %= \sup_{x \in B_R} |x| = R$ and $\tilde{A}:= C_1^{1/q} \, 2^{N-\ell+1} \left(2^q-1\right)^{-1/q}$. 
Hence,
\[J_1 \lesssim \left(\int_{\R^N} \left[ \int_{B_{|x|/2}} |y| |g(y)| \, dy \right]^q |x|^{(\ell-N-1)q} \, d\nu(x)\right)^{1/q} \lesssim \|\nu\|^{1/q}_{0,(N-\ell)q} \int_{\R^N} |g(x)| \, dx.\]

Now for $J_2$, using Minkowski's Inequality we get
\begin{align*}
J_2 %&= \left(\int_{\R^N} \left| \int_{\R^N} K_2(x,y) g(y) \, dy \right|^q d\nu(x)\right)^{1/q}\\
&\leq \int_{\R^N} \left( \int_{\R^N} |K_2(x,y)|^q \, d\nu(x) \right)^{1/q} |g(y)| \, dy.
\end{align*}
It remains to be shown that
\begin{equation} \label{K2 control}
\int_{\R^N} |K_2(x,y)|^q \, d\nu(x) \leq C
\end{equation}
for some constant $C>0$ uniformly on $y$.
%Recall that $K_2(x,y) = K(x,y)-\psi(y/|x|)K(x,0)$. If $2|y|>|x|$, then $\psi(y/|x|)=0$, thus $|K_2(x,y)|=|K(x,y)|$. If $|x| \geq 4|y|$, then $\psi(y/|x|)=1$, thus $|K_2(x,y)|=|K(x,y)-K(x,0)|$. In the region $2|y| \leq |x| < 4|y|$ we have
%$K_2(x,y)=\left[1-\psi\left(\dfrac{y}{|x|}\right)\right]K(x,y) + \psi\left(\dfrac{y}{|x|}\right)[K(x,y)-K(x,0)]$.
For each $y \in \R^{N}$ we get the following upper estimate for the previous integration %in \eqref{K2 control} 
%we may estimate the above  by  %to get the following upper estimate
\begin{align*}
	%\int_{\R^N} |K_2(x,y)|^q \, d\nu(x) &\leq 
	%\begin{multlined}[t]
	%	\int_{|x|<2|y|} |K(x,y)|^q \, d\nu(x)\\
	%	+ \int_{2|y| \leq |x| < 4|y|} \left(|K(x,y)|^q + |K(x,y)-K(x,0)|^q\right) \, d\nu(x)\\
	%	+ \int_{|x| \geq 4|y|} |K(x,y)-K(x,0)|^q \, d\nu(x)
	%\end{multlined}\\
	%&= 
	%\begin{multlined}[t]
		\int_{|x| < 4|y|} |K(x,y)|^q \, d\nu(x)
		+ \int_{|x| \geq 2|y|} |K(x,y)-K(x,0)|^q \, d\nu(x):= \text{\normalfont (I)}+\text{\normalfont (II).}
	%\end{multlined}\\
%	\lesssim %\begin{multlined}[t]
%		\underbrace{\int_{|x| < 4|y|} |x-y|^{(\ell-N)q} \, d\nu(x)}_{\text{\normalfont (I)}}
%		+ \underbrace{\int_{|x| \geq 2|y|} |y|^q \, |x|^{(\ell-N-1)q} \, d\nu(x)}_{\text{\normalfont (II)}}.
	%\end{multlined}
\end{align*}
From conditions \eqref{kd} and \eqref{mu regularity} we have
\begin{align*}
	\text{\normalfont (II)} \lesssim |y|^q \int_{ (B_{2|y|})^c}  \, |x|^{(\ell-N-1)q} \, d\nu(x) 
	%&= |y|^q \sum_{k=1}^\infty \int_{2^k|y| \leq |x| < 2^{k+1}|y|} |x|^{(\ell-N-1)q} \, d\nu(x)\\
	%&\leq |y|^q \sum_{k=1}^\infty (2^k |y|)^{(\ell-N-1)q} \nu(B_{2^{k+1}|y|})\\
	%&\lesssim |y|^{(\ell - N)q} \sum_{k=1}^\infty 2^{k(\ell-N-1)q} \,  (2^{k+1}|y|)^{(N-\ell)q}\\
	%&= 2^{(N-\ell)q} \sum_{k=1}^\infty 2^{-kq}\\
	&\lesssim \|\nu\|_{0,(N-\ell)q} %C_1 \, 2^{(N-\ell)q} \left( 2^q-1 \right)^{-1}
\end{align*}
while from condition \eqref{kc}
\begin{align*}
	\text{\normalfont (I)} & \lesssim \int_{B_{4|y|}} |x-y|^{(\ell-N)q} \, d\nu(x) \\
	 &= \underbrace{\int_{B(y,{|y|/2})} |x-y|^{(\ell-N)q} \, d\nu(x)}_{(I_a)} + \underbrace{\int_{B_{4|y|} \setminus B(y,{|y|/2})} |x-y|^{(\ell-N)q} \, d\nu(x).}_{(I_b)}
\end{align*}
The second part is straightforward:
\begin{align*}
(I_b) \leq \dfrac{1}{(|y|/2)^{(N-\ell)q}} \int_{B_{4|y|}} \, d\nu(x) = \dfrac{\nu(B_{4|y|})}{(|y|/2)^{(N-\ell)q}}
	%&\leq C_{N,\ell,\alpha,q,\eps} \, |y|^{(N-\ell+\alpha)q - (N-\ell)q} 
	\lesssim \|\nu\|_{0,(N-\ell)q}.
\end{align*}
Finally,
%\begin{align*}
%	(I_a) &= \sum_{k=1}^\infty \int_{B(y,\eps^k|y|) \setminus B(y,\eps^{k+1}|y|)} |x-y|^{(\ell-N)q} \, d\nu(x)\\
%	&\leq \sum_{k=1}^\infty \int_{B(y,\eps^k|y|)} (\eps^{k+1}|y|)^{(\ell-N)q} \, d\nu(x)\\
%	&= |y|^{(\ell-N)q} \sum_{k=1}^\infty (\eps^{k+1})^{(\ell-N)q} \, \nu(B(y,\eps^k|y|))\\
%	&\leq C_1 \, |y|^{(\ell-N)q} \sum_{k=1}^\infty (\eps^{k+1})^{(\ell-N)q} \, (\eps^k|y|)^{(N-\ell+\alpha)q}\\
%	&=C_1 \, |y|^{\alpha q} \, \eps^{(\ell-N)q} \sum_{k=1}^\infty \eps^{k \alpha q}\\
%	&=C_1 \, \eps^{(\ell-N)q} \, (\eps^{-\alpha q}-1)^{-1} |y|^{\alpha q}
%\end{align*}
%when $0 < \alpha < 1$. For the case $\alpha = 0$
writing $A_x := \{r \in \R : r > |x-y|\}$ and pointing out that $B(y,|y|/2) \subset B_{2|y|}$, we obtain from \eqref{mu regularity} and {\eqref{mu integral}} 

\begin{align*}
	(I_a) %&%= \int_{B(y,\eps|y|)} (N-\ell)q \left( \int_{|x-y|}^\infty r^{(\ell-N)q-1} \, dr \right) \, d\nu(x)\\
	& = (N-\ell)q \int_{\R^N} \charchi_{B(y,|y|/2)}(x) \left( \int_{0}^\infty \dfrac{\charchi_{A_x}(r)}{r^{(N-\ell)q+1}} \, dr \right) \, d\nu(x)\\
	%& \lesssim \int_{0}^\infty \left( \int_{\R^N} \dfrac{\charchi_{B(y,\eps|y|)}(x) \; \charchi_{A_x}(r)}{r^{(N-\ell)q+1}} \, d\nu(x) \right) \, dr\\
	& = (N-\ell)q  \int_{0}^\infty \left( \int_{B(y,|y|/2) \cap B(y,r)} \dfrac{1}{r^{(N-\ell)q+1}} \, d\nu(x) \right) \, dr\\
	&= (N-\ell)q \left(\int_{0}^{|y|/2} \dfrac{\nu(B(y,r))}{r^{(N-\ell)q}} \, \frac{dr}{r} + \nu(B(y,|y|/2))\int_{|y|/2}^\infty \dfrac{1}{r^{(N-\ell)q+1}} \, dr \right) \\
	%&\leq (N-\ell)q \left[C_2 + \nu(B(y,\eps|y|))\int_{\eps|y|}^\infty r^{(\ell-N)q-1} \, dr \right]\\
	%& \lesssim 1+ \nu(B_{2|y|}) \int_{\eps|y|}^\infty r^{(\ell-N)q-1} \, dr \\
	& \lesssim [[\nu]]_{(N-\ell)q} + \|\nu\|_{0,(N-\ell)q},
	%&= (N-\ell)q \left[C_2 + \nu(B_{2|y|}) \dfrac{1}{(N-\ell)q} \left(\eps|y|\right)^{(\ell-N)q} \right]\\
	%&\leq (N-\ell)q \left[C_2 + \dfrac{C_1 (2|y|)^{(N-\ell)q} (\eps|y|)^{(\ell-N)q}}{(N-\ell)q} \right]\\
	%&= C_2 \, (N-\ell)q  + C_1 \, (2/\eps)^{q(N-\ell)},
\end{align*}
concluding \eqref{K2 control} and thus $\ds{J_2 \lesssim (\,[[\nu]]_{(N-\ell)q} + \|\nu\|_{0,(N-\ell)q})^{1/q} \int_{\R^N} |g(y)| \, dy.} $ \Qed
\end{proof}

%\textcolor{red}{$$ [[\nu]]_{(N-\ell)q}:=\sup_{y \in \R^{N}}W^{|y|/2}_{[N-(N-\ell)q],2} \nu(y) $$}

%\begin{proposition}\label{inf coef const}
%Let $A(D)$ be a homogeneous linear differential operator of order $m<N$ on $\R^{N}$, $N\ge2$, from $E$ to $V$ and  $\mu \in \M_{+}(\R^N,E)$. If $A(D)$ is elliptic and cancelling and $\mu$ satisfies  \eqref{mu regularity} and  \eqref{mu integral} with $q=1$ % the following conditions:
%\begin{enumerate}
%\item[(i)] there exists a constant $C>0$ such that $\mu(B(x,r)) \leq C r^{N-m}$ for every $x \in \R^N$ and $r>0$;
%\item[(ii)] $\ds\int_0^{|y|/4} \dfrac{\mu(B(y,r))}{r^{N-m+1}} \, dr < \infty$,
%\end{enumerate}
%then
%\begin{equation}\label{estim caso inf}
%	\left|\int_{\R^N} u(x) \, d\mu(x) \right| \lesssim \|A(D)u\|_{L^1}, \qquad \forall\, u \in \ccinf(\R^N,E).
%\end{equation}
%\end{proposition}
%%%%%%%%%%%%%%%%%%%%%%%%%%%%%%
\section{Applications and general comments}\label{sec5}

\subsection{Limiting case for trace inequalities for vector fields} \label{trace}

Next we present the validity of the inequality \eqref{raita} for $s=1$ (see \cite[Theorem 1.1]{GRS}) under $(N-1)-$ Ahlfors regularity and an additional uniform {potential} condition on $\nu$.
\begin{theorem} \label{raitathm}
%Let $A(D)$ as Theorem B and $\mu \in \mathcal{M}_{+}(\R^{N},E)$ 
Let $A(D)$ be a homogeneous linear differential operator of order m on $\R^{N}$, $N\ge2$, from $E$ to $F$. Then for all $\nu \in \M_{+}(\R^N)$
satisfying \eqref{mu regularity} and \eqref{mu integral} %$\|\nu\|_{0,N-1}<\infty$ and
%$$\int_{0}^{\epsilon|y|}\frac{|\mu|(B(y,r))}{r^{N-1}}dr \leq C $$
%\begin{equation} \label{potential}
%\int_0^{|y|/2} \dfrac{\nu(B(y,r))}{r^{N}} \, dr \lesssim 1 \quad \text{uniformly on $y$} 
%\end{equation}
%then
 there exists $C>0$ such that 
\begin{equation}\label{raita1}
\int_{\R^N} \left|D^{m-1}u(x)\right| \, d\nu   \leq C %\|\mu \|_{L^{1,(N-1)}} 
		 \|A(D)u\|_{L^{1}}, \quad \forall \, u \in C_{c}^{\infty}(\R^{N},E).
\end{equation}
\end{theorem}

\begin{proof}
The inequality follows by the combination of the identity  
$D^{m-1}u(x)=\int_{\R^{N}}K(x-y)[A(D)u(y)]\,dy$ where $\widehat{K}(\xi):=\sum_{|\alpha|=m-1}\xi^{\alpha}(A^{\ast}\circ A)^{-1}(\xi)A^{\ast}(\xi)$ 
that satisfies \eqref{kc} and \eqref{kd} for $\ell=1$ and then the estimate \eqref{raita1} follows 
%\begin{equation}\nonumber 
 %|K(x-y)| \leq C \; |x - y|^{-N+1} , \quad x \neq y 
%\end{equation}
%and 
%\begin{equation} \nonumber
%	|K(x,y)-K(x,0)| \leq C \; \frac{|y|}{|x|^{N}} , \quad 2|y|\leq |x|,
%\end{equation}
by Lemma \ref{main} for $q=1$, as showed in the proof of inequality \eqref{eqmain}.
\Qed
\end{proof}

%\textcolor{blue}{Pergunta: onde usa que a medida é positiva nessa demonstração?}

As a consequence of the previous proof we can estimate the constant at inequality \eqref{raita1} by
 %$$C \lesssim \|\nu\|_{0,N-1}+\sup_{y \in \R^{N}}\int_{0}^{\textcolor{red}{|y|/2}}\frac{\nu(B(y,r))}{r^{N}}dr. $$
$$C \lesssim \|\nu\|_{0,N-1}+[[\nu]]_{N-1}. $$

\begin{remark}\label{remarkdm}
Let $D^{m}:=(D^{\alpha})_{|\alpha|=m}$ the total derivative operator that is an elliptic and canceling homogeneous linear differential operator. Using \eqref{gaussgreen} follows directly that % from the assumption $\|\mu\|_{N-1}<\infty$ that 
\begin{equation}\label{raita1a}
\int_{\R^N} \left|D^{m-1}u(x)\right| \, d\nu  \lesssim \|\nu\|_{N-1}  %\|\mu \|_{L^{1,(N-1)}} 
		 \|D^{m}u\|_{L^{1}},
\end{equation}
for all  $u \in C_{c}^{\infty}(\R^{N})$ and $\nu \in \M_{+}(\R^N)$. Although the assumption that $\nu$ is $(N-1)-$ Ahlfors regular contrasts with
$\|\nu\|_{0,N-1}<\infty$ at Theorem \ref{raitathm}, the uniform potential condition \eqref{mu integral} is not necessary to the validity of  \eqref{raita1a}.
\end{remark}

%\textcolor{blue}{Veja que aqui a medida é positiva (e escalar) mas é necessário pq usamos as tecnicas que estao na sua dissertação}

In the same spirit of \cite[Theorem A]{HP} the inequality \eqref{raita1} can be extended for the following: 

\begin{theorem}
%Let $A(D)$ as Theorem B and $\mu \in \mathcal{M}_{+}(\R^{N},E)$ 
Let $A(D)$ be a homogeneous linear differential operator of order {$m$} on $\R^{N}$, $N\ge2$, from $E$ to $F$, and assume that  $1 \leq q <\infty$, $0<\ell<N$ and $\ell \leq m$. Then for all $\nu \in \M_{+}(\R^N)$
satisfying \eqref{mu regularity} and \eqref{mu integral} %$\|\nu\|_{0,N-1}<\infty$ and
%$$\int_{0}^{\epsilon|y|}\frac{|\mu|(B(y,r))}{r^{N-1}}dr \leq C $$
%\begin{equation} \label{potential}
%\int_0^{|y|/2} \dfrac{\nu(B(y,r))}{r^{N}} \, dr \lesssim 1 \quad \text{uniformly on $y$} 
%\end{equation}
%then
 there exists $C>0$ such that 
\begin{equation}\label{raita2}
\left(\int_{\R^N} \left| (-\Delta)^{(m-\ell)/2}u(x)\right|^{q} \, d\nu \right)^{1/q}  \leq C  %\|\mu \|_{L^{1,(N-1)}} 
		 \|A(D)u\|_{L^{1}}, \quad \forall \, u \in C_{c}^{\infty}(\R^{N},E).
\end{equation}
\end{theorem}

The proof follows the same steps when proving Theorem \ref{raitathm} and will be omitted.  In particular, the inequality \eqref{raita2} recovers the inequality (1.5) in \cite[Theorem A]{HP} taking $d\mu=|x|^{-N+(N-\ell)q}dx$ for $1 \leq q <N/(N-\ell)$ {(see Remark 4.1). }%\ref{remarkhp}).

\subsection{First order operators}
{
It remains as an open question whether \eqref{stronger m} or \eqref{localcondition} are necessary conditions to obtain a $L^\infty$ solution to \eqref{main1} for homogeneous differential operator $A(D)$ with order $m > 1$. For $m=1$, however, we show that certain (expected) decay regularity on $\mu$ is necessary:
}

\begin{theorem}
	Let $A(D)$ be a first order homogeneous linear differential operator on $\R^{N}$ from $E$ to $F$ %finite dimensional real vector space $E$ to a finite dimensional real vector space $F$ 
and $\mu \in \M_{+}(\R^N,E)$. If there exists $f \in L^\infty(\R^N,F)$ solving \eqref{main1}, then there exists a constant $C>0$ such that 
\begin{equation}	
|\mu(B(x,r))| \leq C r^{N-1} 
\end{equation}
for every $x \in \R^N$ and $r>0$.
\end{theorem}

\begin{proof}
Denoting $A(D) = \ds\sum_{j=1}^N a_j \partial_j$ we have, for every $x \in \R^N$ and almost every $r>0$,
\begin{align*}
	\mu(B(x,r)) &= \int_{B(x,r)} A^*(D)f(y) \, dy
	%= - \sum_{j=1}^N \int_{B(x,r)} a_j \partial_j f(y) \, dy
	= - \sum_{j=1}^N \int_{\partial B(x,r)} a_j^{*} f(y) \dfrac{y_j - x_j}{|y-x|} \, dS(y),
\end{align*}
hence $|\mu(B(x,r))| \leq C_N \|f\|_{L^{\infty}} r^{N-1}$.

To extend this estimate for every $r>0$, let $M \subset \R_+$ be the zero-measure set of values $r>0$ for which the previous estimate does not hold. Given $x \in \R^N$ and $r>0$ we can write $B[x,r] = \cap_{j} B(x,r_j)$, where $(r_j)_j \subset \R_+ \setminus M$ is a decreasing sequence converging to $r$ (note that $\R_+ \setminus M$ is dense in $\R_+$). Thus, simplifying the notation assuming $\mu_{\ell} \in \mathcal{M}_{+}(\R^{N})$ for each $j=1,...,d$ we have
\begin{align*}
	\mu_{\ell}(B(x,r)) % \leq \mu_{\ell}(B[x,r])  % \lim_{j \to \infty} \mu_{\ell}^{\real}(B(x,r_j))
	%&\leq \lim_{j \to \infty} |\mu_\ell(B(x,r_j))|\\
	\leq \lim_{j \to \infty} |\mu(B(x,r_j))|
	\leq C_N \|f\|_{L^{\infty}} \lim_{j \to \infty}  r_j^{N-1}
	&= C_N \|f\|_{L^{\infty}} r^{N-1}.
\end{align*}
Summarizing 
$$ |\mu(B(x,r))| \leq (2d)^{1/2} C_N \|f\|_{L^{\infty}} r^{N-1}.$$
\Qed
%and similarly $\mu_{\ell}^{\imag}(B(x,r)) \leq C_N \|F\|_{L^{\infty}} r^{N-1}$. Therefore,
%\begin{align*}
%	|\mu(B(x,r))|^2 &= \sum_{\ell = 1}^d |\mu_{\ell}(B(x,r))|^2\\
%	&= \sum_{\ell = 1}^d \left[ \mu_{\ell}^{\real}(B(x,r))^2 + \mu_{\ell}^{\imag}(B(x,r))^2 \right]\\
%	&\leq 2d \left(C_N \|F\|_{L^{\infty}} r^{N-1}\right)^2\\
	
%\end{align*}
\end{proof}

\noindent \textbf{Acknowledgment}: The authors want to thank Prof. Pablo de N\'apoli for some discussions on two weighted inequalities. 

%%%%%%%%%%%%%%%%%%%%%%%%%%%%%%%%%%%%%%%%%%%%%%%%%%%%%%%%%%%%%%%%%%%%%%%%%%%%%%%%%%%

\end{document}